\documentclass[a4paper,12pt]{article}

\usepackage[latin1]{inputenc}
\usepackage[T1]{fontenc}
\usepackage[english]{babel}

\usepackage{mathrsfs}
\usepackage{amscd}
\usepackage{amsfonts}
\usepackage{amsmath}
\usepackage{amssymb}
\usepackage{amstext}
\usepackage{amsthm}
\usepackage{amsbsy}

\usepackage{xspace}
\usepackage[all]{xy}
\usepackage{graphicx}
\usepackage{url}
\usepackage{latexsym}

\usepackage{graphicx} 

\usepackage{booktabs} 
\usepackage{array} 
\usepackage{paralist} 
\usepackage{verbatim} 
\usepackage{subfig} 

\usepackage{fancyhdr} 
\usepackage{sectsty}
\allsectionsfont{\sffamily\mdseries\upshape} 

\usepackage[nottoc,notlof,notlot]{tocbibind} 
\usepackage[titles,subfigure]{tocloft} 

\usepackage[textwidth=100pt,textsize=footnotesize,bordercolor=white,color=blue!30]{todonotes}
\usepackage{hyperref} 

\makeatletter
\newcommand*{\rom}[1]{\expandafter\@slowromancap\romannumeral #1@}
\makeatother

\theoremstyle{definition}

\DeclareMathOperator{\card}{card}

\newtheorem{fact}{fact}
\newtheorem{example}[fact]{Example}
\newtheorem{thm}[fact]{Theorem}
\newtheorem{lemma}[fact]{Lemma}
\newtheorem{prop}[fact]{Proposition}
\newtheorem{corollary}[fact]{Corollary}
\newtheorem{defini}[fact]{Definition}

\title{Space-bounded OTMs and REG$^{\infty}$}
\author{Merlin Carl}

\begin{document}

\maketitle


\begin{abstract}
An important theorem in classical complexity theory is that 

\noindent
LOGLOGSPACE=REG, i.e. that languages decidable with double-logarithmic space bound are regular. We consider a transfinite analogue of this theorem.
To this end, we introduce deterministic ordinal automata (DOAs), show that they satisfy many of the basic statements of the theory of deterministic finite automata and regular languages. We then
consider languages decidable by an ordinal Turing machine (OTM), introduced by P. Koepke in 2005 and show that if the working space of an OTM is of strictly smaller cardinality than the input length
for all sufficiently long inputs, the language so decided is also decidable by a DOA.
\end{abstract}

\section{Introduction}

Ordinal Turing machines (OTMs), introduced by P. Koepke in \cite{Ko} and independently by B. Dawson in \cite{Da}, are a well-established and well-studied model of infinitary computability.
Roughly, an OTM is a Turing machine with a tape of proper class size, whose cells are indexed with ordinals, and transfinite ordinal working time.
One attractive feature of OTMs when compared with Infinite Time Turing Machines (ITTMs), introduced in 2000 by J. Hamkins and A. Lewis (\cite{HL}), is that OTMs exhibit a symmetry
between time and space. As a consequence, the complexity theory for OTMs resembles the classical theory much more than it does for ITTMs: For an ITTM, each input is of length
$\omega$, so that e.g. the ITTM-analogue of the class $P$ of polynomial-time computable functions is just the class of functions computable with a constant time bound.
The consideration of complexity theory for OTMs was taken up by B. L\"owe in \cite{L} and then continued by B. L\"owe, B. Rin and the author in \cite{CLR}.

An important result in classical complexity theory is that the class of languages that are decidable by a Turing-machine with double-logarithmic space bound coincides with the
class of regular languages (see e.g. \cite{H}). In this paper, we consider an analogue of this theorem for OTMs. To this end, we introduce deterministic ordinal automata (DOAs) as
a transfinite analogue of deterministic finite automata (DFAs). There are several transfinite variants of DFAs preceeding ours: B\"uchi introduced automata (\cite{Bu}) that work like DFAs and process words
of ordinal length by picking a certain element from the set of states that occured cofinally often during the processing at limit times. Similar models are considered in \cite{StSc} and \cite{HKS}.
A common feature of all these models is that the transition relation is given by a set. In contrast, we allow class-sized transition relation that satisfy a certain mild coherence condition.
It turns out that this condition suffices to carry over much of the theory of DFAs and regular languages (section $2$). To the best of our knowledge, this notion has not been considered before.
We then proceed in section $3$ to define strictly space-bounded OTM-computations as those whose working space is of a strictly smaller cardinality than the input length and show that 
such OTMs in fact work with a constant space bound and that the languages so recognized can be recognized with DOAs.


In the following, $\Sigma$ denotes an alphabet (i.e. a finite set), $\Sigma^{**}$ denotes $\Sigma^{<\text{On}}$, i.e. the set of sequences of ordinal length over $\Sigma$.

\section{REG$^{\infty}$}

We consider ordinal analogues of regular languages. 
In particular, we introduce notions parallelizing in the ordinal realm deterministic finite automata, nondeterministic finite automata, induced congruence relation and
prove analogues of classical theorems like simulation of determinism by nondeterminism and Myhill-Nerode. Though our generalization of a finite automaton may seem to be far too liberal at first, it preserves enough of the
heart of the classical concept to make large parts of the classical theory go through, and much for the same reasons. 

We begin by introducing a notion DOA generalizing deterministic finite automata.

\begin{defini}
 A deterministic ordinal automaton (DOA) is a quintuple $(Q,q_{0},F,D,\Sigma)$ where $Q$ is a set $q_{0}\in Q$, $F\subseteq Q$ and $D:Q\times\Sigma^{**}\rightarrow Q$ is a class function
with the following property: For all $w,w_{1},w_{2}\in\Sigma^{**}$, if $D(q,w)$ is defined and $w=w_{1}w_{2}$, then $D(q,w_{1})$ is also defined and
we have $D(D(q_{0},w_{1}),w_{2})=D(q,w)$.

If $\mathcal{A}=(Q,q_{0},F,D,\Sigma)$ is a DOA, then $S(\mathcal{A}):=\{w\in\Sigma^{**}:D(q_{0},w)\in F\}$ is the language accepted by $\mathcal{A}$.

If $\mathcal{L}\subseteq\Sigma^{**}$ is such that $\mathcal{L}=S(\mathcal{A})$ for some DOA $\mathcal{A}$, then $\mathcal{L}$ is REG$^{\infty}$.

A DOA $\mathcal{A}$ is complete if and only if $D(q,w)$ is defined for all $q\in Q$ and all $w\in\Sigma^{**}$.
\end{defini}

Note that, by this definition, the transition relation $D$ is an arbitrary class, only restricted by the `coherence' or `forgetfulness' condition in the definition; intuitively, by this condition,
the automaton, when in a certain state, has no memory how it got there. It turns out that this rather weak condition seems to lie combinatorially at the heart of many results about regular languages.
We demonstrate this by carrying over some of the main standard results in the theory of regular languages, along with their proofs. The classical counterparts can be found in any
basic textbook on theoretical computer science, such as \cite{Hro}.

\begin{example}
 (i) The language $L_{0}=\{0^{\alpha}1^{\beta}:\alpha,\beta\in\text{On}\}$ is REG$^{\infty}$ for the DOA with $Q=\{z_{1},z_{2}\}$, $q_{0}=z_{1}$, $F=\{z_{2}\}$, $\Sigma=\{0,1\}$, $D(z_{1},0^{\alpha})=z_{1}$ for $\alpha\in\text{On}$,
$D(z_{1},1^{\alpha})=z_{2}$ for $0<\alpha\in\text{On}$ and $D(z_{2},1^{\alpha})=z_{2}$.

(ii) The language $L_{1}=\{0^{\alpha}1^{\alpha}:\alpha\in\text{On}\}$ is not REG$^{\infty}$.
To see this, suppose for a contradiction that $\mathcal{A}=(Q,q_{0},F,D,\Sigma)$ is a DOA with $L_{1}=S(\mathcal{A})$ and consider the words $0^{\alpha}$ for $\alpha<\text{card}(Q)^{+}$.
Since $0^{\alpha}1^{\alpha}\in S_{1}$, $D(q_{0},0^{\alpha})$ must be defined for all $\alpha\in\text{On}$. As $\text{card}{Q}^{+}>\text{card}{Q}$, there must be $\alpha<\beta<\text{card}(Q)^{+}$
such that $D(q_{0},0^{\alpha})=D(q_{0},0^{\beta})$. But this implies $F\ni D(q_{0},0^{\alpha}1^{\alpha})=D(D(q_{0},0^{\alpha}),1^{\alpha})=D(D(q_{0},0^{\beta}),1^{\alpha})=D(q_{0},0^{\beta}1^{\alpha})\notin F$, a contradiction.
($D(q_{0},0^{\beta}q^{\alpha})$ is defined since $D(q_{0},0^{\beta}1^{\beta})$ is defined.)
\end{example}

\begin{prop}{\label{completeDOA}}
For every DOA $\mathcal{A}$, there is a complete DOA $\mathcal{A}^{\prime}$ such that $S(\mathcal{A})=S(\mathcal{A}^{\prime})$.
\end{prop}
\begin{proof}
 This works as in the classical (finitary) case by letting $Q^{\prime}=Q\cup\{q^{\prime}\}$ and letting $D^{\prime}(q,w)=q^{\prime}$ for all $q\in Q^{\prime}$ and all $w\in\Sigma^{**}$ for which $D(q,w)$ is not defined.
\end{proof}

We now consider analogues of non-deterministic finite automata.

\begin{defini}
 An NOA is a quintuple $\mathcal{A}=(Q,q_{0},F,D,\Sigma)$ and $D\subseteq Q\times\Sigma^{**}\times Q$ is a relation such that we have for all $q\in Q$ and all $w,w_{1},w_{2}\in\Sigma^{**}$ with $w=w_{1}w_{2}$ that
$D[D(q,w_{1}),w_{2}]=D(q,w)$. Here, for $X\subseteq Q$ and $w\in\Sigma^{**}$, $D[X,w]$ denotes $\bigcup_{q\in X}D(q,w)$.

If $\mathcal{A}$ is an NOA, then $S(\mathcal{A})=\{w\in\Sigma^{**}:D(q_{0},w)\cap F\neq\emptyset\}$ is the language accepted by $\mathcal{A}$.
\end{defini}

We now imitate the classical power-set construction for simulating non-determinism by determinism in our setting.

\begin{thm}{\label{determinismindeterminism}}
 The languages accepted by some NOA are exactly the languages in REG$^{\infty}$.
\end{thm}
\begin{proof}
Clearly, all REG$^{\infty}$ languages are accepted by some NOA, as every DOA is an NOA.
 
On the other hand, let $S=S(\mathcal{A})$ where $\mathcal{A}=(Q,q_{0},F,D,\Sigma)$ is an NOA.
We construct a DOA $\mathcal{A}^{\prime}:=(Q^{\prime},q_{0}^{\prime},F^{\prime},D^{\prime},\Sigma)$ as follows:
Let $Q^{\prime}=\mathcal{P}(Q)$, the power set of $Q$, $q_{0}^{\prime}=\{q_{0}\}$, $F^{\prime}=\{X\subseteq Q:X\cap F\neq\emptyset\}$
and define $D^{\prime}$ by $D^{\prime}(X,w)=D[X,w](=\bigcup\{D(x,w):x\in X\})$.

We claim that $\mathcal{A}^{\prime}$ is indeed a DOA. So let $X\subseteq Q$, $w_{1},w_{2}\in\Sigma^{**}$. We want to show that $D^{\prime}(X,w_{1}w_{2})=D^{\prime}(D^{\prime}(X,w_{1}),w_{2})$.

``$\subseteq$'': Let $q\in D^{\prime}(X,w_{1}w_{2})$. Then there is $x\in X$ such that $q\in D(x,w_{1}w_{2})$. As $\mathcal{A}$ is an NOA,
we have $D(x,w_{1}w_{2})=D[D(x,w_{1}),w_{2}]$, so $q\in D[D(x,w_{1}),w_{2}]$. Let $q^{\prime}\in D(x,w_{1})$ such that $q\in D(q^{\prime},w_{2})$.
By definition of $D^{\prime}$, we certainly have $D^{\prime}(X,w)\subseteq D^{\prime}(Y,w)$ whenever $X\subseteq Y\subseteq Q$. Hence (as $\{x\}\subseteq X$ and
$D^{\prime}(\{x\},w_{1})=D(x,w_{1})$) we have $q^{\prime}\in D^{\prime}(X,w_{1})$ and hence (since $\{q^{\prime}\}\subseteq D^{\prime}(X,w_{1})$ and
$D^{\prime}(\{q^{\prime}\},w_{2}=D(q^{\prime},w_{2})$) we have $q\in D^{\prime}(D^{\prime}(X,w_{1}),w_{2})$. Since $q$ was arbitrary, this
shows that $D^{\prime}(X,w_{1}w_{2})\subseteq D^{\prime}(D^{\prime}(X,w_{1}),w_{2})$.

``$\supseteq$'': Let $q\in D^{\prime}[D^{\prime}(X,w_{1}),w_{2}]$. Then there is $q^{\prime}\in D^{\prime}(X,w_{1})$ such that $q\in D(q^{\prime},w_{2})$. 
Furthermore there is $x\in X$ such that $q^{\prime}\in D(x,w_{1})$. Thus $q\in D[D(x,w_{1}),w_{2}]$. But now we have $D(x,w_{1})\subseteq D^{\prime}(X,w_{1})$
and therefore $q\in D[D(x,w_{1}),w_{2}]\subseteq D[D^{\prime}(X,w_{1}),w_{2}]=D^{\prime}(X,w_{1}w_{2})$. As $q$ was arbitrary, this shows that
$D^{\prime}(D^{\prime}(X,w_{1}),w_{2})\subseteq D^{\prime}(X,w_{1}w_{2})$.

\bigskip

Hence $\mathcal{A}^{\prime}$ is indeed a DOA. We finish by showing that $S(\mathcal{A}^{\prime})=S(\mathcal{A})$:

``$\subseteq$'': Let $w\in S(\mathcal{A}^{\prime})$. Hence $D^{\prime}(\{q_{0}\},w)\in F^{\prime}$, i.e. $D^{\prime}(\{q_{0}\},w)\cap F\neq\emptyset$.
Since $D^{\prime}(\{q_{0}\},w)=D[\{q_{0}\},w]=D(q_{0},w)$, we have $D(q_{0},w)\cap F\neq\emptyset$, hence $w\in S(\mathcal{A})$.

``$\supseteq$'': Let $w\in S(\mathcal{A})$. Hence $D(q_{0},w)\cap F\neq\emptyset$. Thus $F\cap D^{\prime}(\{q_{0}\},w)=D(q_{0},w)\cap F\neq\emptyset$,
which implies $D^{\prime}(\{q_{0}\},w)\in F^{\prime}$, i.e. $w\in S(\mathcal{A}^{\prime})$.

\end{proof}

We turn to an analogue of Myhill-Nerode.

\begin{defini}
 Let $\mathcal{L}\subseteq\Sigma^{**}$. The congruence relation on $\Sigma^{**}$ induced by $\mathcal{L}$ is defined as follows: For $w_{1},w_{2}\in\Sigma^{**}$,
we have $w_{1}\equiv_{\mathcal{L}}w_{2}$ if and only if, for all $w\in\Sigma^{**}$, we have $w_{1}w\in\mathcal{L}\leftrightarrow w_{2}w\in\mathcal{L}$.

If the class of $\equiv_{\mathcal{L}}$-equivalence classes is a set (more formally: if there is a set $X$ such that every $w\in\Sigma^{**}$
is $\mathcal{L}$-equivalent to some element of $X$), we say that $\mathcal{L}$ satisfies the `ordinal Myhill-Nerode condition' (MH for short).
\end{defini}

\begin{thm}{\label{ordinalmyhillnerode}}
 $\mathcal{L}\subseteq\Sigma^{**}$ is REG$^{\infty}$ if and only if $\mathcal{L}$ satifies MH.
\end{thm}
\begin{proof}
Suppose first that $\mathcal{L}$ is REG$^{\infty}$, and let, by Proposition \ref{completeDOA} $\mathcal{A}=(Q,q_{0},F,D,\Sigma)$ be a complete DOA such that $\mathcal{L}=S(\mathcal{A})$.
For $q\in Q$ let $Z_{\mathcal{A}}(q):=\{w\in\Sigma^{**}:D(q_{0},w)=q\}$. Then, for each $q\in Q$, the elements of $Z_{\mathcal{A}}(q)$ are pairwise $\equiv_{\mathcal{L}}$-equivalent:
For $w_{1},w_{2}\in Z_{\mathcal{A}}(q)$ and $w\in\Sigma^{**}$, we have $D(D(q_{0},w_{1}),w)=D(q,w)=D(D(q_{0},w_{2}),w)$, hence either both $D(D(q_{0},w_{1}),w)$ and $D(D(q_{0},w_{2}),w)$
belong to $\mathcal{L}$ or neither does, i.e. $w_{1}w\in\mathcal{L}\leftrightarrow w_{2}w\in\mathcal{L}$. Since $w$ was arbitrary, we have $w_{1}\equiv_{\mathcal{L}}w_{2}$.

Thus, all the $Z_{A}(q)$ are subclasses of $\equiv_{\mathcal{L}}$-equivalence classes and, as $\mathcal{A}$ is complete, every $w\in\Sigma^{**}$ belongs to one of the $Z_{\mathcal{A}}(q)$.
Thus every $\equiv_{\mathcal{L}}$-equivalence class is a union of some (at least one) $Z_{\mathcal{A}}(q)$, hence there are at most as many $\equiv_{\mathcal{L}}$-equivalence classes
as there are elements in $Q$, i.e. only set-sized many.

\bigskip

Now let $\mathcal{L}\subseteq\Sigma^{**}$ be such that $\equiv_{\mathcal{L}}$ has only set-sized many equivalence classes on $\Sigma^{**}$. Pick a representative from each equivalence class
and denote by $\mathcal{C}$ their collection; for $w\in\Sigma^{**}$, denote by $[w]_{\mathcal{L}}$ the element of $\mathcal{C}$ equivalent to $w$. 
We construct a DOA $\mathcal{A}$ with $S(\mathcal{A})=S$ as follows:

Let $Q=\mathcal{C}$, $q_{0}=[\varepsilon]_{\mathcal{L}}$ (where $\varepsilon$ denotes the empty word), $F=\{[w]_{\mathcal{L}}:w\in\mathcal{L}\}$.

Note that, for all $w\in\Sigma^{**}$, we either have $[w]_{\mathcal{L}}\subseteq\mathcal{L}$ or $[w]_{\mathcal{L}}\cap\mathcal{L}=\emptyset$: For if we have $w_{1},w_{2}\in[w]_{\mathcal{L}}$,
then $w_{1}\equiv_{\mathcal{L}}w_{2}$, so $\mathcal{L}\ni w_{1}=w_{1}\varepsilon\leftrightarrow\mathcal{L}\ni w_{2}\varepsilon=w_{2}$.

Now define $D$ by setting $D([w_{1}]_{\mathcal{L}},w_{2})=[w_{1}w_{2}]_{\mathcal{L}}$. It is easy to check that $\mathcal{A}=(Q,q_{0},F,D,\Sigma)$ is as desired.
 
\end{proof}

\begin{corollary}{\label{booleanclosure}}
 REG$^{\infty}$ is closed under complementation, union and intersection.
\end{corollary}
\begin{proof}
Let $\mathcal{L}_{1}$, $\mathcal{L}_{2}$ be REG$^{\infty}$.

For complementation, consider a complete DOA $\mathcal{A}=(Q,q_{0},D,F,\Sigma)$ for $\mathcal{L}_{1}$ and let $\mathcal{A}^{\prime}=(Q,q_{0},D,Q\setminus F,\Sigma)$; it is easy to see that $S(\mathcal{A}^{\prime})=\mathcal{L}_{1}$.

For unions, let $\mathcal{A}_{1}=(Q_{1},q_{0,1},D_{1},F_{1},\Sigma)$, $\mathcal{A}_{2}=(Q_{2},q_{0,2},D_{2},F_{2},\Sigma)$ be DOAs such that $S(\mathcal{A}_{1})=\mathcal{L}_{1}$ and $S(\mathcal{A}_{2})=\mathcal{L}_{2}$.
Assume without loss of generality that $Q_{1}$ and $Q_{2}$ are disjoint. Form an NOA $\mathcal{A}=(Q,q_{0},D,F,\Sigma)$ by letting $Q=((Q_{1}\cup Q_{2})\setminus\{q_{0,1},q_{0,2}\})\cup\{q_{0}\}$ (where $q_{0}$ is neither
contained in $Q_{1}$ nor in $Q_{2}$), $F=F_{1}\cup F_{2}$ and defining $D(q_{0},w)=D_{1}(q_{0,1},w)\cup D_{2}(q_{0,2},w)$ and $D(q,w)=D_{i}(q,w)$ for $q\in Q_{i}$ ($i\in\{1,2\}$). If $q_{0,1}\in F_{1}$ or $q_{0,2}\in F_{2}$,
then we also put $q_{0}$ in $F$.
It is easy to see that $S(\mathcal{A})=\mathcal{L}_{1}\cup\mathcal{L}_{2}$. By Theorem \ref{determinismindeterminism}, $\mathcal{L}_{1}\cup\mathcal{L}_{2}$ is REG$^{\infty}$.

Closure under intersection follows by de Morgan's rules from closure under complementation and union.

\end{proof}



\begin{corollary}{\label{closureconsequences}}
 For $w\in\{0,1\}^{**}$, let $w_{0}$, $w_{1}$ denote the subsequences consisting of the $0$s and $1$s in $w$, respectively. Furthermore, let $|w|_{0}$ and $|w|_{1}$ denote the cardinality
of the set of places in $w$ taken by $0$ or $1$. Then the following two languages are not REG$^{\infty}$:

(1) $\mathcal{L}=\{w\in\Sigma^{**}:\text{otp}(w_{0})=\text{otp}(w_{1})\}$

(2) $\mathcal{L}^{\prime}=\{w\in\Sigma^{**}:|w|_{0}=|w|_{1}\}$
\end{corollary}
\begin{proof}
It is easy to see that the language $\mathcal{L}^{\prime\prime}:=\{0^{\alpha}1^{\beta}:\alpha,\beta\in\text{On}\}$ is REG$^{\infty}$: The corresponding DOA has two states $s_{0}$ and $s_{1}$,
both of which are accepting and of which $s_{0}$ is the starting state. The transition relation is given by $D(s_{0},0^{\alpha})=s_{0}$, $D(s_{0},0^{\alpha}1^{\beta})=s_{1}$, $D(s_{1},1^{\alpha})=s_{1}$
for all $\alpha,\beta\in\text{On}$.

(1) By Corollary \ref{booleanclosure}, if $\mathcal{L}$ was REG$^{\infty}$, then so was $\mathcal{L}\cap\mathcal{L}^{\prime\prime}=\{0^{\alpha}1^{\alpha}:\alpha\in\text{On}\}$; but we saw above that this is not the case.

(2) If $\mathcal{L}^{\prime}$ was REG$^{\infty}$, then intersecting with $\mathcal{L}^{\prime\prime}$ would yield the $\infty$-regularity of $\mathcal{L}_{c}:=\{0^{\alpha}1^{\beta}:\text{card}(\alpha)=\text{card}(\beta)\}$.
Let $\mathcal{A}=(Q,q_{0},D,F,\{0,1\})$ be a DOA with $S(\mathcal{A})=\mathcal{L}_{c}$. By the pigeonhole principle, there are two infinite cardinals $\kappa<\lambda$ with $D(q_{0},0^{\kappa})=D(q_{0},0^{\lambda})$.
But then $F\not\ni D(q_{0},0^{\kappa}1^{\lambda})=D(D(q_{0},0^{\kappa}),1^{\lambda})=D(D(q_{0},0^{\lambda}),1^{\lambda})=D(q_{0},0^{\lambda}1^{\lambda})\in F$, a contradiction.
\end{proof}

We also get a rather straightforward analogue of the pumping lemma. As the proof is the same, we prove something slightly stronger, which is closer to one direction of an analogue of Jaffe's theorem:

\begin{defini}
For $w\in\Sigma^{**}$ and $\alpha,\beta<|w|$, let $v=w\upharpoonright[\alpha,\beta]$ be the interval of $w$ from index $\alpha$ to index $\beta$. 

Moreover, let $w_{(\alpha)}=w\upharpoonright[0,\alpha)$ and $w^{(\alpha)}=w\upharpoonright[\alpha,|w|)$.
\end{defini}

\begin{thm}{\label{pumping}}
Let $S$ be REG$^\infty$, $\mathcal{A}=(Q,q_{0},F,D,\Sigma)$ a DOA with $S(\mathcal{A})=S$. Let $w\in\Sigma^{**}$ be sufficiently long, more specifically $|w|>\text{card}(Q)$. 
Then there are $\alpha,\beta<|w|$ such that, for all $i\in\omega$, $w_{(\alpha)}(w\upharpoonright[\alpha,\beta])^{i}w^{(\alpha)}\equiv_{S}w$.
\end{thm}
\begin{proof}
By the pigeonhole-principle and the fact that $|w|>\text{card}(Q)$, there are $\alpha,\beta<|w|$ such that $D(q_{0},w_{(\alpha)})=D(q_{0},w_{(\beta)})$.
Let $v=w\upharpoonright[\alpha,\beta]$. Then $D(q_{0},w_{(\alpha)})=D(q_{0},w_{(\alpha)}+v^{i})$ for all $i\in\omega$, so
$w_{(\alpha)}\equiv_{S}w_{(\alpha)}+v^{i}$, so $w_{(\alpha)}(vw^{(\beta)})\equiv_{S}w_{(\alpha)}v^{i}w^{(\beta)}$.
\end{proof}


\bigskip
\noindent
\textbf{Remark}: The proof clearly allows us to demand that $v$ is `short', i.e. that $\beta-\alpha\leq\text{card}(Q)^{+}$. 

Note that the proof does not show
that we can repeat $w\upharpoonright[\alpha,\beta]$ also $\gamma$ many times for $\gamma\in\text{On}$: For example, a DOA with to states $s_{0},s_{1}$ may satisfy $D(s_{0},1^{i})=s_{0}$ for any $i\in\omega$
but also $D(s_{0},1^{\omega})=s_{1}$. 
In fact, this stronger version is false, as we will now show. 

\begin{defini}
Let $\mathcal{L}_{\omega-\text{rep}}$ be the language consisting of those elements of $\{0,1\}^{**}$ that have a subsequence of the form $\underbrace{www...}_{\omega\times}$, where $w\in\{0,1\}^{**}$.
\end{defini}


\begin{prop}
$\mathcal{L}_{\omega-\text{rep}}$ is REG$^{\infty}$.
\end{prop}
\begin{proof}
The DOA for deciding $\mathcal{L}_{\omega-\text{rep}}$ has two states $q_{0}$ and $q_{1}$.  $q_{1}$ is the only accepting state, $q_{0}$ is the starting state. 
The transition function sends $D(q_{0},w)$ to $q_{1}$ if and only if $w\in\mathcal{L}_{\omega-\text{rep}}$ and to $q_{0}$, otherwise and it sends $D(q_{1},w)$ to $q_{1}$ for any $w$.
It is easy to see that this is a DOA that decides $\mathcal{L}_{\omega-\text{rep}}$.
\end{proof}

\begin{lemma}
There is a language $\mathcal{L}\subseteq\{0,1\}^{**}$ that is REG$^{\infty}$, and for all $\alpha$, there is some $w\in\{0,1\}^{**}$ of length $>\alpha$ such there are no $w_{0},w_{1},w_{2}\in\{0,1\}^{**}$ with $w=w_{0}w_{1}w_{2}$ and $w_{0}w_{1}^{\omega}w_{2}\in\mathcal{L}$. 
\end{lemma}
\begin{proof}
Let $\mathcal{L}=\overline{\mathcal{L}_{\omega-\text{rep}}}$ be the complement of $\mathcal{L}_{\omega-\text{rep}}$. Then $\mathcal{L}$ is REG$^{\infty}$ as the complement of a language that is REG$^{\infty}$. We claim that $\mathcal{L}$ contains arbitrarily long words. This suffices, as by definition no word of the form $w_{0}w_{1}^{\omega}w_{2}\in\mathcal{L}$ can belong to $\mathcal{L}$.

It is easy to see that, if $w:\alpha\rightarrow 2$ is generic over $L$ for the forcing consisting of finite functions from $\alpha$ to $2$ ordered by inclusion, then $w$ will be as desired, as the set of conditions $c$ that do not extend to any function $f:\alpha\mapsto 2$ where the $\iota$th position starts a repetition of a word of length $\delta$ is dense for all $\iota,\delta<\alpha$. To avoid the sledgehammer of forcing and also the extra assumptions necessary to guarantee the existence of generic objects, we offer the following more direct construction:

Let us define the transfinite Morse-Thue-sequence MT$_{\infty}$ $(s_{\iota}:\iota\in\text{On})$ as follows: $s_{0}=0$, $s_{\iota+1}=s_{\iota}\overline{s_{\iota}}$ (where $\overline{s}$ denotes the word that has $0$s where $s$ has $1$s and vice versa) and for a limit ordinal $\lambda$, $s_{\lambda}$ is defined by $s_{\lambda}(\iota)=s_{\iota+1}(\iota)$ for all $\iota<\lambda$. We claim that no element of MT$_{\infty}$ has a subsequence of the form $www$, which is clearly much stronger than what we require.\footnote{For the elements of finite index, this is well-known, see, e.g., the entry in the Online Encyclopedia of Integer Sequences \url{https://oeis.org/A010060}.} Let us denote by MT$_{\infty}(\iota)$ the $\iota$-th element of this sequence and by MT$_{\infty}\upharpoonright[\iota,\xi)$ the sequence restricted to the indices $\iota$ up to $\xi$, for $\iota,\xi\in\text{On}$. Clearly, for each $\iota\in\text{On}$,  MT$_{\infty}\upharpoonright[\omega\cdot\iota,\omega\cdot(\iota+1))$ will either be MT$_{\infty}\upharpoonright[0,\omega)$ (which is just the classical Morse-Thue sequence) or its ``complement'' $\overline{\text{MT}_{\infty}\upharpoonright[0,\omega)}$. 
 Thus, no \textit{finite} subword of the form $www$ can appear in MT$_{\infty}$. 

It is easy to see from the definition of MT$_{\infty}$ that, for all $\alpha,\gamma\in\text{On}$, we have MT$_{\infty}(\omega^{\gamma}\cdot\alpha)=$MT$_{\infty}(\alpha)$ ($\ast$). Now suppose for a contradiction that, for some word $w$, the substring $www$ appears somewhere in MT$_{\infty}$. Let $\alpha:=|w|$ be the length of $w$, and let us write $\alpha$ in Cantor normal form to the base $\omega$ as $\alpha=\omega^{\gamma_{0}}\cdot k_{0}+\omega^{\gamma_{1}}\cdot k_{1}+...+\omega^{\gamma_{n}}\cdot k_{n}$ with $\gamma_{0}>\gamma_{1}>...>\gamma_{n}$ and $k_{0},...,k_{n}\in\omega$. Moreover, let $\rho$ be the index at which the subword $www$ starts in MT$_{\infty}$ for the first time and write $\rho$ in the form $\omega^{\gamma_{0}}\cdot\beta+\bar{\rho}$ with $\beta\in\text{On}$ and $\bar{\rho}<\omega^{\gamma_{0}}$. By elementary properties of ordinal addition, we have that 
$\rho+\omega^{\gamma_{0}}\cdot k=\omega^{\gamma_{0}}\cdot(\beta+k)$ for all $k\in\omega$. 

Let $w^{\prime}$ be the word consisting of the elements of $w$ with index of the form $\omega^{\gamma_{0}}\cdot i$, $0\leq i\leq k_{0}$. Then $w^{\prime}w^{\prime}w^{\prime}$ is also the sequence of elements of MT$_{\infty}$ with indices of the form $\rho+\omega^{\gamma_{0}}\cdot i$, with $0< i\leq 3k_{0}$. Now $\rho+\omega^{\gamma_{0}}\cdot i=\omega_{\gamma^{0}}\cdot(\beta+i)$ for these $i$, so by observation ($\ast$) above, this coincides with MT$_{\infty}\upharpoonright[\beta+1,\beta+3k_{0}+1)$, which is thus a finite subsequence of consecutive elements of MT$_{\infty}$ of the form $w^{\prime}w^{\prime}w^{\prime}$, a contradiction.


\end{proof}

Thus, the transfinite analogue of the pumping lemma already fails for ``$\omega$-pumping''.



We now consider $\infty$-regularity for unary languages, i.e. languages over an alphabet with only one element.

\begin{prop}{\label{sets are regular}}
  For each $\alpha\in\text{On}$ and each $X\subseteq\alpha$, $\{1^{\beta}:\beta\in X\}$ is REG$^{\infty}$. In fact, whenever $S\subseteq\Sigma^{**}$ is a set (for an arbitrary $\Sigma$),
then $S$ is REG$^{\infty}$. 
\end{prop}
\begin{proof}
 The appropriate DOA has one state for each initial segment of an element of $S$ and connects them in the obvious manner.
\end{proof}

\begin{defini}
 For $X\subseteq\text{On}$ and $s$ a symbol, $s^{X}$ abbreviates $\{s^{\alpha}:\alpha\in X\}$.
\end{defini}


\begin{lemma}{\label{nonregunary}}
Neither of the following unary languages is REG$^{\infty}$: (1) $\mathcal{L}_{1}=\{1^{\kappa}:\kappa\in\text{Card}\}$, (2) $\mathcal{L}_{2}=\{1^{\omega^{\alpha}}:\alpha\in\text{On}\}$, (3) $\mathcal{L}_{3}=\{1^{\alpha^{2}}:\alpha\in\text{On}\}$.
\end{lemma}
\begin{proof}
All three proofs work by contradiction. For $i\in\{1,2,3\}$, let $\mathcal{A}_{i}$ be a DOA with $S(\mathcal{A}_{i})=\mathcal{L}_{i}$; the start state is always denoted $q_{0}$, the transition by $D$, the set of accepting states by $F$ etc.
In the following, $+$ always denotes ordinal addition.

(1) As $\mathcal{A}_{1}$ only has a set of states, but there is a proper class of cardinals, by the pigeonhole principle there must be $\kappa<\lambda\in\text{Card}$ such that
$D(q_{0},\kappa)=D(q_{0},\lambda)\in F$. Now we have $\lambda+\lambda=\lambda2\notin\text{Card}$ and $\kappa+\lambda=\lambda\in\text{Card}$.
It follows that $F\ni D(q_{0},1^{\lambda})=D(q_{0},1^{\kappa+\lambda})=D(D(q_{0},1^{\kappa}),1^{\lambda})=D(D(q_{0},1^{\lambda}),1^{\lambda})=D(q_{0},1^{\lambda2})\notin F$, a contradiction.

The proofs for (2) is similar, noting that $\omega^{\alpha}+\omega^{\beta}=\omega^{\beta}$ is a power of $\omega$ for $\alpha<\beta$, while $\omega^{\alpha}2$ is never a power of $\omega$.

For (3), consider the sequence given by $\alpha_{0}=\omega$, $\alpha_{\iota+1}=\omega^{\alpha}$, $\alpha_{\lambda}=\bigcup_{\iota<\lambda}\alpha_{\iota}$ for $\lambda$ a limit ordinal.
It is easy to see that $\alpha_{\iota}^{2}+\alpha_{\iota^{\prime}}^{2}=\alpha_{\iota^{\prime}}^{2}$ for $\iota<\iota^{\prime}$ is always a square, while $\gamma^{2}\cdot 2$ is never a square or an ordinal for $\gamma\in\text{On}$. 
\end{proof}



\begin{prop}{\label{countinglanguage}}
 The language $\mathcal{L}_{\text{count}}:=\{\circ(0^{\iota}1:\iota<\alpha):\alpha\in\text{On}\}$ consisting of words of the form
$1$, $101$, $101001$, $1010010001...10^{\gamma}10^{\gamma+1}1...1$ is not REG$^{\infty}$ (here, $\circ$ denotes concatenation of words).
\end{prop}
\begin{proof}
 Otherwise, let $\mathcal{A}=(Q,q_{0},D,F,\{0,1\})$ be a DOA with $S(\mathcal{A})=\mathcal{L}_{\text{count}}$. Denote $\circ(0^{\iota}1:\iota<\alpha)$ by $w_{\alpha}$.

Let $q\in Q$ be arbitrary. There must some $\alpha(q)\in\text{On}$ such that the following holds: For every $\beta>\alpha(q)$, there is $\gamma>\beta$
such that $D(q,0^{\beta})=D(q,0^{\gamma})$. Thus, every state that is reachable from $q$ via a sequence of more than $\alpha$ zeroes is reachable
by such sequences of arbitrarily high length. Let $\sigma:=\text{sup}\{\alpha(q):q\in Q\}$.

Now consider $D(q_{0},w_{\sigma}1)=:\hat{q}$. By definition of $\sigma$, $D(\hat{q},0^{\sigma})=D(\hat{q},0^{\beta})$ for arbitrarily large $\beta>\sigma$. Pick such a $\beta$.
Then $F\ni D(q_{0},w_{\sigma}10^{\sigma})=D(\hat{q},0^{\sigma})=D(\hat{q},0^{\beta})=D(q_{0},w_{\sigma}10^{\beta})\notin F$, a contradiction.

\end{proof}


For the main result of this paper in the next section, we will also need an ordinal version of $\lambda$-NFAs:

\begin{defini}{\label{def lambda}}
Fix a special symbol $\lambda$. 
From now on, we will assume that $\Sigma$ never contains $\lambda$. If $w\in\Sigma^{**}$, a $\lambda$-enrichment of $w$ is defined as a sequence in $(\Sigma\cup\{\lambda\})^{**}$ in which
the subsequence of elements of $\Sigma$ is exactly $w$.

A $\lambda$-NOA with alphabet $\Sigma$ is simply an NOA with the alphabet $\Sigma\cup\{\lambda\}$. If $\mathcal{A}=(Q,q_{0},D,F,\Sigma)$ is a $\lambda$-NOA, then
$L(\mathcal{A})$ is the set of $w\in\Sigma^{**}$ such that $D(q_{0},w^{\prime})\cap F\neq\emptyset$ for some $\lambda$-enrichment $w^{\prime}$ of $w$.

A language $\mathcal{L}\subseteq\Sigma^{**}$ is $\lambda$-REG$^{\infty}$ if and only if there is a $\lambda$-NOA $\mathcal{A}$ with $L(\mathcal{A})=\mathcal{L}$.
\end{defini}

\begin{lemma}{\label{lambda reg}}
 A language $\mathcal{L}\subseteq\Sigma^{**}$ is $\lambda$-REG$^{\infty}$ if and only if it is REG$^{\infty}$.
\end{lemma}
\begin{proof}
 Clearly, if $\mathcal{L}$ is REG$^{\infty}$, it is also $\lambda$-REG$^{\infty}$, as every DOA is also a $\lambda$-NOA (where all transitions for words containing $\lambda$ are undefined).

On the other hand, let $\mathcal{L}=L(\mathcal{A})$, where $\mathcal{A}=(Q,q_{0},F,D,\Sigma)$ is a $\lambda$-NOA. Then we define an NOA $\mathcal{A}^{\prime}=(Q,q_{0},F,D^{\prime},\Sigma)$
as follows: For $w\in\Sigma^{**}$ and $q\in Q$, $D^{\prime}(q,w)=\bigcup\{D(q,w^{\prime}):w^{\prime}\text{ is a }\lambda-\text{enrichment of }w\}$. It is easy to see that this defines an NOA
and that $L(\mathcal{A})=L(\mathcal{A}^{\prime})$.
\end{proof}

\bigskip
\noindent
\textbf{Remark}: We also haven't considered ordinal versions of Mealy or Moore automata, but we encourage the interested reader to do so.

\section{Space-Bounded OTMs}

We now work towards our main result. To this end, we define the space complexity of an OTM-program. This concept was introduced by L\"owe in \cite{L}.

\begin{defini}
 Let $f:\text{On}\rightarrow\text{On}$ be a function and $P$ an OTM-program. $P$ belongs to SPACE$^{\infty}(f)$ if and only if there is an ordinal $\beta$ such that, whenever $w$ is a word of length $\alpha>\beta$,
the computation of $P^{w}$ uses only the first $\beta f(\alpha)$ many cells of the scratch tape.
\end{defini}

A classical theorem in complexity theory is that, if the space usage $s$ of a Turing machine $T$ is such that $2^{2^{s(n)}}\leq c\cdot n$ for some $c\in\mathbb{N}$, then $T$ in fact has a constant bound on its space usage and hence 
decides a regular language. (See e.g. \cite{H} for a proof of this.)

We work towards an infinitary version of this. In the following, let $f:\text{On}\rightarrow\text{On}$ be a (class) function such that $\text{card}(f(\alpha))<\text{card}(\alpha)$ for all sufficiently large $\alpha$,
i.e. $f$ `lowers cardinalities'. We will begin by showing that, if $P$ belongs to SPACE$^{\infty}(f)$ for such an $f$, then $P$ in fact belongs to SPACE$^{\infty}(1)$, i.e. there is a uniform constant bound on the
amount of cells $P$ uses on the scratch tape.

\begin{thm}{\label{minimalword}}
 Let $P$ be an OTM-program, $\kappa$ a cardinal, and let $w$ be a $0$-$1$-word of minimal length such that $P^{w}$ uses $\kappa$ many scratch tape cells. Then $\text{card}(|w|)\leq\kappa$.
\end{thm}
\begin{proof}
 Let $|w|=\delta$. Moreover, let $\sigma$ denote the order type of the set of scratch tape cells used in the computation of $P^{w}$ (thus $\text{card}(\sigma)=\kappa$).
Form the elementary hull $H$ of $\sigma+1\cup\{w\}$ in $H_{\delta^{++}}$ (the set of sets heriditarily of cardinality $\leq\delta^{++}$, where $\alpha^{+}$ denotes the cardinal successor of $\alpha$). 
Note that $H$ will in particular contain the computation of $P^{w}$. Form the transitive collapse $M$ of $H$, and denote by $\bar{w}$ the image of $w$ under the collapsing map. 
Then $\sigma+1\subseteq M$ and $M\models$`$P^{\bar{w}}$ uses a set of scratch tape cells or order type $\sigma$'. Furthermore, we have $|M|=\aleph_{0}\kappa=\kappa$,
so the length of $\bar{w}$ has cardinality at most $\kappa$. Since $M$ is transitive, the computation of $P^{\bar{w}}$ in $M$ is the same as that in $V$.
Thus $P^{\bar{w}}$ uses already a set of scratch tape cells of order type $\sigma$ and hence of cardinality $\kappa$.
\end{proof}

\textbf{Remark}: It is easy to see and well-known that a halting OTM-computation on an input of length $\alpha$ can only have a length of cardinality $\leq\text{card}(\alpha)$. Thus,
we can in fact replace $\text{card}(|w|)\leq\kappa$ with $\text{card}(|w|)=\kappa$ in the theorem statement.

\begin{defini}
 Call an OTM-program $P$ `strictly space-bounded' if and only if there is a function $f:\text{On}\rightarrow\text{On}$ such that $\text{card}(f(\alpha))<\text{card}(\alpha)$ for all sufficiently large $\alpha$ and $P$ belongs to SPACE$^{\infty}(f)$.
\end{defini}

\begin{corollary}{\label{constantbound}}
If $P$ is strictly space-bounded, then $P$ belongs to SPACE$^{\infty}(1)$.
\end{corollary}
\begin{proof}
 Assume otherwise. This means that, for each $\kappa\in\text{Card}$, there is a word $w$ such that $P^{w}$ uses at least $\kappa$ many scratch tape cells. By Theorem \ref{minimalword}, $|w|\leq\kappa$.
But, by the definition of $f$, if $\kappa$ is sufficiently large, then $\text{card}(f(\kappa))<\kappa$, so since $P$ belongs to SPACE$^{\infty}(f)$, $P^{w}$ uses less than $\kappa$ many scratch tape cells, 
a contradiction.
\end{proof}


Our final goal is to show that, for each OTM-program $P$ with a constant use of scratch tape, there is a $\lambda$-NOA accepting exactly those words for which $P$ halts. The proof will use an ordinal version
of crossing sequences.

\begin{defini}
 Let $P$ be an OTM-program, and let $w=(w_{\iota}:\iota<\alpha)\in\{0,1\}^{**}$. Consider the computation of $P^{w}$ and let $\beta<\alpha$ be an ordinal.
The interval $[\beta,\beta+\omega)$ is called the $\beta$-block of $w$. The block-crossing sequence $\text{bcs}(\beta)$ associated with $\beta$
is the sequence of quintuples $(s,t,\rho,i,f)$, called $P$-snippets, where in the $\iota$th tuple $(s_{\iota},t_{\iota},i_{\iota})$, $s_{\iota}$ is the inner state of $P$ when the reading 
head is in the $\beta$-block of $w$ for the $\iota$th time in the course of the computation,
and likewise $t_{\iota}$ is the content of the scratch tape, $\rho$ is the position of the read/write-head on the scratch tape, $i\in\omega$ is the relative position of the reading head in the block at this time and $f$ is the content
of the first $i$ many bits of the $\omega$-block in which the reading head is currently located.

For the sake of simplicity, we will from now on regard $\rho$ and $s$ as `absorbed' into $t$ (e.g. via a special mark on the tape) and work with triples $(s,i,f)$ instead of quintuples as $P$-snippets.
\end{defini}

We will eventually construct the desired NOA whose states will be the possible candidates for the block-crossing sequences. We begin by showing that the possible block-crossing sequences for such a program $P$ as above form a set. 
(Note that this is not trivial, since
the possible inputs are a proper class.) If $P$ has constant scratch tape use bounded by $\gamma$, then the second components of the triples $(s,i,f)$ will
belong to the set $^{\gamma}2$; the first component is an element of a finite set and the last component is an element of $\omega$. Thus, there is a set $T$ of triples that can 
possibly occur in a block crossing sequence. It remains
to control the length of such a sequence. This is our next goal. 

We start by recalling the following looping criterion for infinitary machines, which can e.g. be found in \cite{HL} for the case of ITTMs:

\begin{prop}{\label{spaceboundedlooping}}
 Let $P$ be an OTM-program. Suppose in the computation of $P$, a configuration $c$ is repeated such that, for all times between the two occurences of $c$,
every other configuration had all components (tape contents) at least as large as the corresponding component in $c$. Then $P$ never halts.
\end{prop}
\begin{proof}
(Sketch) After arriving at $c$ for the second time, the same steps will be repeated, so $c$ will occur a third, fourth etc. time.
The only way to escape the loop would be at a limit time. However, the condition above ensures that the configuration at any limit time which is preceeded by cofinally many occurences of $c$ will be $c$ as well.
\end{proof}

\begin{lemma}{\label{resettozero}}
Let $w=(w_{\iota}:\iota<\alpha)\in\{0,1\}^{**}$. In the course of a halting computation $P^{w}$ with scratch tape use bounded by $\gamma$, the reading head will be positioned on $w_{0}$ less than $(2^{\card{\gamma}})^{+}$ many times.
\end{lemma}
\begin{proof}
 First, note that the sequence $(c_{\iota}:\iota<\delta)$ of machine configurations (i.e. the inner states and the scratch tape contents) occuring at times when the reading head is positioned at $w_{0}$ is continuous in the sense
that $(q_{\lambda},s_{\lambda})=\text{liminf}_{\iota<\lambda}(s_{\iota},t_{\iota})$ for limit ordinals $\lambda<\delta$. This is due to the behaviour of OTMs at limit stages and in particular the fact that the reading head position at 
limit times is the inferior limit of the sequence of earlier reading head positions, i.e. at a limit of times at which the reading head was on $w_{0}$, the reading head will again be at $w_{0}$.

Now assume otherwise and consider the sequence of the first $(2^{\text{card}(\gamma)})^{+}$ many such configurations. 
Let $\iota^{*}$ be the index in the computation of $P^{w}$ at which the reading head is at $w_{0}$ for the $(2^{\text{card}(\gamma)})^{+}$th time.

If every configuration occurs only boundedly often before this time, then for each such configuration, the set of suprema of the indices of their occurences
would be majorized by some $c_{\iota}$, and the set of these $\iota$ would be a cofinal subset of $(2^{\text{card}(\gamma)})^{+}$ of cardinality $2^{\text{card}(\gamma)}$, contradicting the regularity of the successor cardinal $(2^{\text{card}(\gamma)})^{+}$.

Thus, there is some $\alpha<\iota^{*}$
such that all configurations occuring after time $\alpha$ in the computation occur cofinally often before $\iota^{*}$. Again by the regularity of $(2^{\text{card}(\gamma)})^{+}$,
there is $\lambda<\iota^{*}$ 
which is simultaneously for each of these configuration a limit of the indices at which these configurations occur and a time at which the reading head is at $w_{0}$. 
As $\lambda>\alpha$, $c_{\lambda}$ will occur cofinally often
below $\iota^{*}$. Also, by the liminf-rule, $c_{\lambda}$ will in each component be less than or equal to every other configuration occuring cofinally often before $\iota^{*}$.
But this implies that the computation is strongly looping after the first two such occurences of $c_{\lambda}$, which contradicts the assumption that $P^{w}$ halts.
\end{proof}

\begin{corollary}{\label{rarevisits}}
Let $\beta<\alpha$. In the course of the computation of $P^{w}$, there will be at most $2^{\text{card}(\gamma)}$ many disjoint time intervals at which the reading head is in the $\beta$th block.
\end{corollary}
\begin{proof}
 Note that, by the rules for moving the reading head, a block other than the $0$th block can only be reached ``from the left'', while the $0$th block can only be entered when the reading head is moved to the left from some limit position.
Therefore, whenever the reading head is in the $0$th block, it must have been on $w_{0}$ before without leaving the $0$th block in the meantime. Thus, the $0$th block can only be visited $<2^{\card{\gamma}})^{+}$, i.e. 
$\leq 2^{\text{card}(\gamma)}$ many times by Lemma \ref{resettozero}.

As any other block can only be entered from the left, the reading head must have been in the $0$th block before entering such a block anew, thus the same holds for all other blocks.
\end{proof}

It remains to control how long the reading head can remain in a block.

\begin{lemma}{\label{remaining}}
 A time interval in which the reading head remains in the same $\omega$-block can have no more than $2^{\text{card}(\gamma)}$ many elements.
\end{lemma}
\begin{proof}
 Suppose for a contradiction that the reading head remains in some block for $(2^{\text{card}(\gamma)})^{+}$ many steps. Shifting the time interval if necessary, we assume that this starts at computation time $0$. 
If every of the possible
$\omega$ many reading head positions in this episode occured only boundedly often, $(2^{\text{card}(\gamma)})^{+}$ would have cofinality $\omega$, a contradiction. 
Thus, there is $\alpha<(2^{\text{card}(\gamma)})^{+}$ such that every reading head position occuring after time $\alpha$ occurs cofinally often before time $(2^{\text{card}(\gamma)})^{+}$.
Let $k$ be the minimal element of the set of these positions. Then the same argument as for Lemma \ref{resettozero} shows that the reading head can be on position $k$ at most $2^{\text{card}(\gamma)}$ many times
during the interval. But this leads to a cofinal subset of $(2^{\text{card}(\gamma)})^{+}$ with cardinality $\leq 2^{\text{card}(\gamma)}$, a contradiction.
\end{proof}

\begin{corollary}{\label{headpositions}}
For each block, the reading head position is inside this block at most $2^{\text{card}(\gamma)}$ many times during the computation of $P^{w}$.
\end{corollary}
\begin{proof}
 By Corollary \ref{rarevisits}, the reading head enters each block at most $2^{\text{card}(\gamma)}$ many times and by Lemma \ref{remaining} remains there for at most $2^{\text{card}(\gamma)}$ many steps,
thus the reading head is inside the block at most $2^{\text{card}(\gamma)}2^{\text{card}(\gamma)}=2^{\text{card}(\gamma)}$ many times.
\end{proof}

\begin{corollary}{\label{crossinglength}}
 There are at most $2^{2^{\text{card}(\gamma)}}$ many possible block crossing sequences.
\end{corollary}
\begin{proof}
By Corollary \ref{headpositions}, each such sequence has a length of cardinality at most $2^{\text{card}(\gamma)}$; moreover, there are at most $\omega^{2}2^{\text{card}(\gamma)}$ many possible entries
in such a sequence. Thus, the number of such sequences is bounded by $(2^{\text{card}(\gamma)})^{2^{\text{card}(\gamma)}}=2^{\text{card}(\gamma)2^{\text{card}(\gamma)}}=2^{2^{\text{card}(\gamma)}}$.
\end{proof}

We will now construct an NOA that accepts exactly those words $w\in\{0,1\}^{**}$ for which $P^{w}$ halts. As announced above, the states of this NOA will be the possible block crossing sequences.

We assume that the input word $w$ has a mark lh on its left side and another mark rh to its right and that $P$ notices (via special states) when rh is reached.
Furthermore, we assume that $P$ starts with the reading head on lh and never goes back there in the course of the computation.

We change the definition of a block-crossing sequence slightly: We don't need the whole sequence of machine configurations while the reading head
is on the block, it suffices to have the sequence of machine configurations for the time points when the reading head enters the block anew (i.e. after having been in another block in the meantime).
(Clearly, as such sequences are obtained from the others by deleting some elements, they are not longer and hence the estimates above remain valid.)

\begin{thm}{\label{spaceboundregular}}
 Let $P$ be an OTM-program such that the scratch space usage of the computation of $P^{w}$ for every $w\in\{0,1\}^{**}$ is bounded by a constant $\gamma$. Then there is
a $\lambda$-NOA $\mathcal{A}_{P}$ such that $S(\mathcal{A}_{P})=\{w\in\{0,1\}^{**}:P^{w}\downarrow\}=:S(P)$.
\end{thm}
\begin{proof}

Assume without loss of generality that, if $P^{w}$ halts, then it halts with the reading head on rh. This can easily be arranged by changing $P$ slightly to $P^{\prime}$ which works
like $P$, but when $P$ assumes a halting state, moves the reading head to the right until it reaches rh and stops then.

We pick a starting state $q_{0}$. Now let $S_{P}$ be the set of all $P$-snippets in which the scratch tape is empty, the inner state of $P$ is the initial state
and the reading and read/write heads are on position $0$. Furthermore, let $\text{Seq}_{P}$ be the set of all sequences of $P$-snippets of length $\leq2^{\text{card}(\gamma)}$ whose first elements belongs
to $S_{P}$. The $\lambda$-NOA we are about to build will have disjoint deterministic components $\mathcal{A}_{z}$, one for each $z\in\text{Seq}_{P}$. The idea is that $z$ is a guess for the sequence of snippets
when $P$ has the reading head at the start of the input tape. Thus the first step of $\mathcal{A}$ is a $\lambda$-transition to some element of $\text{Seq}_{P}$.
Hence, for each $z\in\text{Seq}_{P}$, we introduce a state $q_{z}$ to the states of $\mathcal{A}$ and add $\lambda$-transitions $D(q_{0},\lambda)=\{q_{z}:z\in\text{Seq}_{P}\}$ to the class of
transitions of $\mathcal{A}$.

We now define the components $\mathcal{A}_{z}$ separately. Let $z\in\text{Seq}_{P}$ be given. The states of $\mathcal{A}_{z}$ are all sequences of $P$-snippets of length $\leq2^{\text{card}(\gamma)}$.
We proceed to define the transition relation $D_{z}$ of $\mathcal{A}_{z}$.

First, let $\textbf{q}=(s,i,f)$ be a $P$-snippet, $w\in\{0,1\}^{**}$. Then $d_{z}(\textbf{q},w)$, the local $w$-successor of $\textbf{q}$, is determined as follows:
Add $w$ to the right of $f$. Then simulate $P$ from this situation on, with the input given by $fw$. If the reading head is set back to $0$ before every
bit of $w$ has been read, we let $d_{z}(\textbf{q},w)=\emptyset$. Otherwise, we consider the $P$-snippet $c$ describing the situation in which the reading head 
arrives at a cell to the right of $w$ for the first time and let $d_{z}(\textbf{q},w)=c$.

Next, we define $D_{z}(s,w)$, for $s\in\text{Seq}_{P}$. Let $s=(s_{\iota}:\iota\leq2^{\text{card}(\gamma)})$. 
Then $D_{z}(s,w)=(d_{z}(s_{\iota},w):\iota\leq2^{\text{card}(\gamma)})$, if it holds for all $s_{\iota}$ with $d_{z}(s_{\iota},w)=\emptyset$ that the $(\iota+1)$th entry of $z$
is the $P$-snippet after the reading head is set back to $0$; otherwise, we leave $D_{z}(s,w)$ undefined. This condition enforces that the $P$-snippets `match'
the initial situation $z$.

The set $F_{z}$ of accepting states of $\mathcal{A}_{z}$ is the set of sequences that have a final entry $t$ in which the inner state is an accepting state of $P$.
This defines $\mathcal{A}_{z}$, and thus $\mathcal{A}$.

\bigskip

Let $\mathcal{A}_{P}=(Q,q_{0},F,D,\{0,1\})$. We show that $\mathcal{A}_{P}$ is indeed a $\lambda$-NOA. But this is easy to see, as the indeterminism in fact
only happens in the first step of the form $D(q_{0},\lambda)$, while for $q\neq q_{0}$, $D(q,w)$ has only one element.

\bigskip

Finally, we show that $S(\mathcal{A}_{P})=S(P)$. First, if $P^{w}$ halts, 
then $D(q_{0},w)$ contains the sequence $\bar{s}$ describing the actual behaviour of $P^{w}$ on $w_{(\omega)}$, the sequence of the first $\omega$ many bits of $w$. Now, if $\bar{w}$ denotes
the rest of $w$, then $D(\bar{s},\bar{w})$ will be the crossing sequence describing the actual behaviour of $P^{w}$ on rh,
which, by the assumption that $P^{w}$ halts, has as its last element a triple whose first component is a halting state and all of whose
other entries must have an rh-state as their first component.

On the other hand, if $w\in S(\mathcal{A}_{P})$, this means that $D(q_{0},w)$ contains an element of $F$, which by definition of $D$
means that there is a computation of $P^{w}$ that halts.

\end{proof}

\begin{corollary}{\label{strictlyreginfty}}
Every language $\mathcal{L}\subseteq\{0,1\}^{**}$ that is decidable by a strictly space-bounded OTM is REG$^{\infty}$.
\end{corollary}
\begin{proof}
 Combine Theorem \ref{spaceboundregular} and Lemma \ref{lambda reg}. 
\end{proof}

Putting Corollary \ref{strictlyreginfty} and our negative results on REG$^{\infty}$ together, we obtain:

\begin{corollary}
 None of the following languages is semi-decidable by a strictly space-bounded OTM:
\begin{enumerate}
 \item $S_{1}:=\{0^{\alpha}1^{\alpha}:\alpha\in\text{On}\}$
 \item $S_{2}:=\{1^{\omega^{\alpha}}:\alpha\in\text{On}\}$
 \item $S_{3}:=\{1^{\alpha^{2}}:\alpha\in\text{On}\}$
 \item $S_{4}:=\{\circ(0^{\iota}1:\iota<\alpha):\alpha\in\text{On}\}$
 \item $S_{5}:=\{w\in\{0,1\}^{**}:\text{otp}(w_{0})=\text{otp}(w_{1})\}$
\end{enumerate}
\end{corollary}
\begin{proof}
 
\end{proof}

This shows in particular that strictly-space bounded OTMs are strictly weaker than e.g. linearly space-bounded OTMs, as there is obviously an OTM-program in SPACE$^{\infty}(\text{id})$ that decides $\{0^{\alpha}1^{\alpha}:\alpha\in\text{On}\}$
by simply writing the $0$s to the scratch tape, then going back to the start of the scratch tape and replacing $0$s by $1$s.




We cannot expect to turn Corollary \ref{strictlyreginfty} into a strict equivalence, as Proposition \ref{sets are regular} shows that any sub\textbf{set} of $\{0,1\}^{**}$ is REG$^{\infty}$, while not every such set will in general 
be OTM-computable.

\section{Conclusion and further work}

We have defined a transfinite version of regularity that generalizes the classical version in a rather straightforward manner. As it turns out, the theory is in large parts parallel to the classical theory, even though the relevant realm is considerably extended. As a by-product, our generalization allows us to see more clearly which structural features of regular languages play a role in proving their basic properties. A somewhat weak spot is the lack so far of a satisfying analogue of the pumping lemma that allows for ``transfinite pumping''. Whether such an analogue exists and what the exact formulation should be is currently an open question. 

Another question is whether there is any way to get an equivalence result from Corollary \ref{strictlyreginfty}.  One approach might be to strengthen the notion of an OTM: In his Master's thesis \cite{Le}, E. Lewis\footnote{Not to be confused with A. Lewis, one of the authors of \cite{HL}.} introduced OTMs with infinite programs, in which programs are (possibly infinite) sets of Turing commands and the states are indexed with ordinals; computations are then defined exactly as for OTMs with finite programs. One of the results of \cite{Le} is that these `Lewis machines' are equivalent to OTMs with a (set-sized) oracle. Our definitions of space complexity, strictly space-boundedness etc. clearly apply to Lewis-machines. Another approach might be to  consider more restrictive versions of regularity, where the DOA is required to be OTM-computable within certain complexity restrictions. Whether these approaches will lead to a non-trivial result will be taken up in future work. 

Given the fact that regular languages play an important role in computer science since there are many natural examples of regular languages, it will be interesting to see whether there are examples of REG$^{\infty}$-classes corresponding to natural classes in mathematics. If so, REG$^{\infty}$ could serve as a complexity measure for such objects.



It is then natural to ask for transfinite generalizations of further stages of the Chomsky hierarchy, which we plan to take up in future work.


\begin{thebibliography}{}
 \bibitem[Bu]{Bu} J.R. B\"uchi. Decision methods in the theory of ordinals. Bull. Amer. Math. Soc.71, pp. 767-770 (1965)
 \bibitem[CLR]{CLR} M. Carl, B. L\"owe, B. Rin. Koepke Machines and Satisfiability for Infinitary Propositional Languages. In: J. Kari, F. Manea, I. Petre (eds.), Unveiling Dynamics and Complexity, $13$th Conference on Computability in Europe, CiE 2017,
Turku, Finland, Proceedings, LNCS 10307, pp. 187-197 (2017)
 \bibitem[Da]{Da} B. Dawson. Ordinal time turing computation. PhD thesis, Bristol (2009)
 \bibitem[H]{H} K.A. Hansen. Computational Complexity Theory. Lecture 2: Space Complexity. Lecture Notes. Available online: \url{https://users-cs.au.dk/arnsfelt/CT08/scribenotes/lecture2.pdf} (2008)
 \bibitem[HaLe]{HL} J. D. Hamkins, A. Lewis. Infinite Time Turing Machines. Journal of Symbolic Logic 65(2), 567-604 (2000)
 \bibitem[HKS]{HKS} M. Huschenbett, A. Kartzow and P. Schlicht. Pumping for Ordinal-Automatic Structures. Computability, 6, pp. 125-164 (2017)
 \bibitem[Hro]{Hro} J. Hromkovic. Theoretische Informatik. Springer Vieweg (2011)
 \bibitem[Ko]{Ko} P. Koepke. Turing computations on ordinals. Bulletin of Symbolic Logic 11, 377-397 (2005)
 \bibitem[L]{L} B.\ L\"owe. Space bounds for infinitary computation. In: A.\ Beckmann, U.\ Berger, B.\ L\"owe, J.\ V.\ Tucker (eds.), 
Logical Approaches to Computational Barriers, Second Conference on Computability in Europe, CiE 2006, Swansea, UK, Proceedings, LNCS 3988, pp.\ 319--329 (2006)
\bibitem[Le]{Le} E. Lewis. Computation with Infinite Programs. MSc Thesis, Universiteit van Amsterdam (2018)
\bibitem[StSc]{StSc}  P. Schlicht, F. Stephan. Automata on ordinals and automaticity of linear orders. Annals of Pure and Applied Logic 5, pp. 523-527 (2013)
\end{thebibliography}
\end{document}